%% file: main.tex
\documentclass[reqno]{amsart}
\usepackage{amsmath}
\usepackage{amsfonts}
\usepackage{latexsym}
\usepackage{mathtools}
\usepackage{amsthm}
\usepackage{amssymb}
\usepackage{tikz}
\numberwithin{equation}{section}
\newcommand{\cat}{\text{cat}}

\newtheorem{theorem}{Theorem}[section]
\newtheorem*{theorem*}{Theorem}
\newtheorem{lemma}[theorem]{Lemma}
\newtheorem{definition}[theorem]{Definition}
\newtheorem{proposition}[theorem]{Proposition}
\newtheorem{corollary}[theorem]{Corollary}
\newtheorem{remark}[theorem]{Remark}

\newtheorem*{acknowledgments*}{Acknowledgments}
\usepackage[hidelinks]{hyperref}
\usepackage{docmute}

\title[LS Category of Lee Forms on LCS Manifolds]{Lusternik--Schnirelmann category of Lee Forms on locally conformally symplectic manifolds}
\author{Kenji Fukushi}
\address{Department of Mathematics, Graduate School of Science, Kyoto University, Sakyo-ku, Kyoto 606-8502, Japan}
\email{fukushi.kenji.85n@st.kyoto-u.ac.jp}

\date{}
\begin{document}

\begin{abstract}
\sloppy 
The Lusternik--Schnirelmann category of a closed symplectic manifold admits various estimates. Locally conformally symplectic (LCS) geometry is a generalization of symplectic geometry involving a closed one-form, called the Lee form. In this paper, we introduce Farber's Lusternik--Schnirelmann category for closed one-forms into LCS geometry by applying it to the Lee class. In particular, this provides a new topological approach to LCS structures of the second kind. We then obtain lower bounds for this invariant under LCS blow-ups.

\end{abstract}

\maketitle

\input{intro.tex}

\input{preliminaries.tex}

\input{lscatoflee.tex}

\input{LCSmanifoldswithpositiveLeecategory.tex}

\input{lowerbound.tex}

\input{blowup.tex}

\section*{Acknowledgments}
The author would like to thank his supervisor, Tsuyoshi Kato, for his valuable advice and continuous support.
The author used ChatGPT (OpenAI) for English-language polishing and readability improvements. The resulting text was reviewed and edited by the author, who takes full responsibility for the manuscript.

\end{document}

%% file: intro.tex
\section{Introduction}

The Lusternik--Schnirelmann category (LS category) is a classical homotopy invariant. Throughout this paper, we use the unnormalized convention. All topological spaces are assumed to be connected unless otherwise stated.

\begin{definition}
Let $X$ be a topological space. The Lusternik--Schnirelmann category $\cat(X)$ is the least positive integer $k$ such that $X$ can be covered by $k$ open subsets $U_1,\ldots,U_k$ for which each inclusion $U_i\hookrightarrow X$ is null-homotopic.
\end{definition}

For a path-connected finite-dimensional CW complex $X$, the classical cohomological and dimensional estimates give
\[
\operatorname{cup\text{-}length}(X)+1\leq \cat(X)\leq \dim X+1;
\]
see, for example, \cite{CLOT}. Here $\operatorname{cup\text{-}length}(X)$ denotes the largest integer $r$ for which there exist positive-degree cohomology classes $u_1,\ldots,u_r$ such that
\[
u_1\smile\cdots\smile u_r\neq0.
\]

These estimates are particularly effective for symplectic manifolds. Let $(M^{2n},\omega)$ be a closed symplectic manifold. Since $[\omega]^n\neq0\in H^{2n}(M;\mathbb R)$, the cup-length estimate gives $\cat(M)\geq n+1$. Under stronger topological assumptions, this lower bound can be substantially improved. Rudyak and Oprea proved that if the symplectic class vanishes on the image of the Hurewicz homomorphism $\pi_2(M)\to H_2(M;\mathbb Z)$, then the LS category attains its maximal value \cite{RO}. In our convention, their result reads $\cat(M)=2n+1$. In particular, this holds when $\pi_2(M)=0$, and hence for aspherical symplectic manifolds.

Locally conformally symplectic geometry provides a natural setting in which the usual symplectic cup-product argument is no longer directly available. An LCS manifold is a triple $(M^{2n},\Omega,\theta)$, where $\Omega$ is a nondegenerate two-form and $\theta$ is a closed one-form satisfying
\[
d\Omega=\theta\wedge\Omega.
\]
The one-form $\theta$ is called the \emph{Lee form}; see, for example, \cite{B2}. In contrast with the symplectic case, $\Omega$ is in general not closed and therefore does not determine a class in ordinary de Rham cohomology. Consequently, although $\Omega^n$ is a volume form, the ordinary cup-length argument arising from a symplectic cohomology class is no longer available.

The behavior of the ordinary LS category is correspondingly quite different in the LCS setting. While every closed symplectic $2n$-manifold satisfies $\cat(M)\geq n+1$, closed LCS $2n$-manifolds realize every value
\[
3,4,\ldots,2n+1
\]
of the ordinary LS category. More strongly, every value in this range is realized by an LCS manifold of the first kind; see Definition~\ref{deffirstkind} and Proposition~\ref{propcatlcs}. Thus the ordinary LS category of a closed LCS manifold has no dimension-dependent lower bound beyond $3$. Moreover, even maximal ordinary LS category gives no obstruction to an LCS structure being of the first kind.

This motivates the use of a refinement of LS category associated with a closed one-form. Farber introduced such an invariant \cite{F1,F2}, in connection with Morse--Novikov theory \cite{N}. For a finite CW complex $X$ and a cohomology class $\xi\in H^1(X;\mathbb R)$, one obtains an invariant $\cat(X,\xi)$, which agrees with the ordinary Lusternik--Schnirelmann category when $\xi=0$.

Roughly speaking, $\cat(X,\xi)$ measures the number of categorical subsets required to cover $X$ after allowing one additional part of the space to be moved arbitrarily far in the negative direction with respect to a closed one-form representing $\xi$. More precisely, for every $N>0$, one allows a subset to be deformed so that the integral of the closed one-form along the deformation decreases by at least $N$, while the remaining part is covered by categorical subsets. We recall the precise definition in Section~2.

This construction parallels the passage from classical Morse theory to Morse--Novikov theory. For a closed manifold $M$, classical Lusternik--Schnirelmann theory gives
\[
\#\operatorname{Crit}(f)\geq \cat(M)
\]
for every smooth function $f\colon M\to\mathbb R$; see, for example, \cite{CLOT}. Morse--Novikov theory replaces exact one-forms $df$ by arbitrary closed one-forms, while Farber's invariant provides a corresponding extension of LS category to cohomology classes represented by closed one-forms.

For an LCS manifold $(M,\Omega,\theta)$, the Lee form determines a distinguished cohomology class $[\theta]\in H^1(M;\mathbb R)$, and hence a distinguished invariant $\cat(M,[\theta])$. Morse--Novikov and Lichnerowicz--Novikov cohomologies associated with the Lee form have been studied in both locally conformally symplectic and locally conformally K\"ahler geometry; see, for example, \cite{AOT,LV,OV,YZ}. In contrast, the category $\cat(M,[\theta])$ appears not to have been systematically studied in the LCS setting.

The category associated with the Lee class gives an obstruction that the ordinary LS category does not detect. If a cohomology class $\xi$ on a closed manifold admits a nowhere-vanishing closed one-form representative, then $\cat(M,\xi)=0$. As recalled after Definition~\ref{deffirstkind}, the Lee form of an LCS structure of the first kind is nowhere vanishing. Hence every LCS structure of the first kind satisfies $\cat(M,[\theta])=0$. Therefore, $\cat(M,[\theta])>0$ implies that every LCS structure with Lee class $[\theta]$ is of the second kind. Thus positivity of $\cat(M,[\theta])$ gives a topological obstruction, depending only on the Lee class, to the existence of an LCS structure of the first kind with that Lee class. This is in sharp contrast with the ordinary LS category $\cat(M)$, which may take any possible value even among LCS manifolds of the first kind.

We first apply this observation to several classes of LCS manifolds. Products of closed surfaces provide examples for which the category can be computed explicitly using results of Farber and Sch\"utz \cite{FS1}. We also construct four-dimensional examples with positive category by combining topological properties of the underlying manifold with the existence theorem of Bertelson and Meigniez for conformal symplectic structures with prescribed nonzero Lee class \cite{BM2}.

The main result of this paper concerns LCS blow-ups. Let $(M,\Omega,\theta)$ be a closed LCS manifold with $[\theta]\neq0$, let $i\colon Z\hookrightarrow M$ be a nonempty compact induced globally conformally symplectic submanifold of real codimension $2r$, where $r\geq2$, and let
\[
\pi\colon\widetilde M\longrightarrow M
\]
be the blow-up along $Z$. By the blow-up theorem of Yang, Yang and Zhao, $\widetilde M$ admits an LCS structure with Lee class $\pi^*[\theta]$ \cite{YYZ}. We prove that
\[
\cat(\widetilde M,\pi^*[\theta])\geq r-1.
\]
More generally, if there exist positive-degree classes $u_1,\ldots,u_s\in H^{>0}(M;\mathbb C)$ whose product restricts nontrivially to $Z$, then
\[
\cat(\widetilde M,\pi^*[\theta])\geq r-1+s.
\]
The proof combines the cohomological estimate of Farber and Sch\"utz \cite{FS1} with the topology of the exceptional divisor.

Since $r\geq2$, the basic lower bound is already positive. It follows that the pullback class $\pi^*[\theta]$ admits no nowhere-vanishing closed one-form representative. Consequently, every LCS structure on $\widetilde M$ whose Lee class is $\pi^*[\theta]$ is of the second kind. This conclusion depends only on the pullback Lee class and not on the particular LCS structure produced by the blow-up construction.

When the original LCS structure is of the first kind, the blow-up therefore produces a jump in the category associated with the Lee class:
\[
\cat(M,[\theta])=0,
\qquad
\cat(\widetilde M,\pi^*[\theta])\geq r-1>0.
\]
Thus blowing up along any nonempty compact induced globally conformally symplectic submanifold of real codimension at least four changes the Lee class from one admitting a nowhere-vanishing closed one-form representative to one for which such a representative is impossible.

Finally, we compare this obstruction with the Euler characteristic. A nonzero Euler characteristic already excludes LCS structures of the first kind, since their Lee forms are nowhere vanishing. We construct a four-dimensional example for which a point blow-up satisfies $\chi(\widetilde M)=0$ but $\cat(\widetilde M,\pi^*\xi)>0$. Thus the Euler characteristic gives no obstruction in this example, whereas the category associated with the Lee class still excludes LCS structures of the first kind with Lee class $\pi^*\xi$. This shows that the category associated with a closed one-form can detect obstructions to first-kind LCS structures that are invisible to the Euler characteristic.

%% file: Preliminaries.tex
\section{Preliminaries}
\subsection{Lusternik--Schnirelmann category for closed one-forms}

The Lusternik--Schnirelmann category associated with a closed one-form is a homotopy-theoretic invariant, and it is therefore useful to formulate the theory on spaces more general than smooth manifolds. Following Farber, we use the notion of a continuous closed one-form on a topological space \cite{F2,FS2}. This allows the theory to be defined for finite CW complexes and makes the dependence on the pair $(X,\xi)$ explicit.

\begin{definition}
Let $X$ be a topological space. A \emph{continuous closed one-form} $\omega$ on $X$ is represented by an open cover $\mathcal U=\{U\}$ of $X$ together with a collection of continuous functions
\[
f_U\colon U\longrightarrow\mathbb R,
\qquad U\in\mathcal U,
\]
such that, for every $U,V\in\mathcal U$, the difference
\[
f_U|_{U\cap V}-f_V|_{U\cap V}
\]
is locally constant on $U\cap V$. Two such collections $\{f_U\}_{U\in\mathcal U}$ and $\{g_V\}_{V\in\mathcal V}$ are said to be equivalent if their union satisfies the same condition on the cover $\mathcal U\cup\mathcal V$. A continuous closed one-form is an equivalence class of such collections.
\end{definition}

A continuous function $f\colon X\to\mathbb R$ determines a continuous closed one-form, denoted by $df$, by taking the cover consisting of $X$ itself and the single local potential $f$. Thus this notion extends the usual notion of an exact one-form. On a smooth manifold, an ordinary smooth closed one-form determines a continuous closed one-form by taking local primitives.

Continuous closed one-forms can be integrated along continuous paths. If $\omega$ is represented by $\{f_U\}_{U\in\mathcal U}$ and $\gamma\colon[0,1]\to X$ is a continuous path, choose a subdivision
\[
0=t_0<t_1<\cdots<t_m=1
\]
such that $\gamma([t_i,t_{i+1}])\subset U_i$ for some $U_i\in\mathcal U$. Then one defines
\[
\int_\gamma\omega
=
\sum_{i=0}^{m-1}
\bigl(
f_{U_i}(\gamma(t_{i+1}))
-
f_{U_i}(\gamma(t_i))
\bigr).
\]
This integral is independent of the choices involved and depends only on the homotopy class of $\gamma$ relative to its endpoints \cite{F2,FS2}. Integration over closed loops determines a cohomology class $[\omega]\in H^1(X;\mathbb R)$. Conversely, for a finite CW complex every class $\xi\in H^1(X;\mathbb R)$ can be represented by a continuous closed one-form.

Let now $X$ be a finite CW complex and let $\xi\in H^1(X;\mathbb R)$. Fix a continuous closed one-form $\omega$ representing $\xi$. Farber's Lusternik--Schnirelmann category associated with $\xi$, further studied by Farber and Sch\"utz, is defined as follows \cite{F2,FS1}.

\begin{definition}
Let $N>0$. A subset $A\subset X$ is called \emph{$N$-movable with respect to $\omega$} if there exists a homotopy
\[
h_t\colon A\longrightarrow X,\qquad t\in[0,1],
\]
such that $h_0$ is the inclusion and, for every $x\in A$,
\[
\int_{\gamma_x}\omega\leq -N,
\]
where $\gamma_x(t)=h_t(x)$ for $t\in[0,1]$.
\end{definition}

Thus an $N$-movable subset can be continuously moved by an amount at least $N$ in the negative direction measured by $\omega$.

\begin{definition}[{\cite[Definition~5.3]{FS1}}]
The invariant $\cat(X,\xi)$ is the least integer $k\geq0$ such that, for every $N>0$, there exists an open cover
\[
X=F\cup F_1\cup\cdots\cup F_k
\]
such that $F$ is $N$-movable with respect to $\omega$ and each inclusion $F_i\hookrightarrow X$, $i=1,\dots,k$, is null-homotopic.
\end{definition}

The invariant $\cat(X,\xi)$ depends only on the cohomology class $\xi$ and not on the chosen representative $\omega$, and it is a homotopy invariant of the pair $(X,\xi)$ \cite{FS1}. If $\xi=0$, then $\cat(X,0)=\cat(X)$ \cite{FS1}. For a closed one-form $\omega$, we shall also write $\cat(X,[\omega])$.

For a broader account of Lusternik--Schnirelmann theory for closed one-forms and its applications to topology and dynamics, we refer to \cite{FS2}. Farber and Sch\"utz also studied the related invariant $\cat^1(X,\xi)$ and developed lower bounds using homological category weights \cite{FS3}. In this paper, we work exclusively with $\cat(X,\xi)$.

\subsection{Locally conformally symplectic manifolds}

We recall the basic definitions of locally conformally symplectic geometry; see, for example, \cite{B2}. Throughout this paper, we use the convention
\[
d\Omega=\theta\wedge\Omega.
\]

\begin{definition}
Let $M$ be a smooth manifold of dimension $2n$, where $n\geq2$. A \emph{locally conformally symplectic structure} on $M$ is a nondegenerate two-form $\Omega$ for which there exists a closed one-form $\theta$ satisfying
\[
d\Omega=\theta\wedge\Omega.
\]
The one-form $\theta$ is called the \emph{Lee form}, and the triple $(M,\Omega,\theta)$ is called a \emph{locally conformally symplectic manifold}, or an \emph{LCS manifold}.
\end{definition}

The Lee form is uniquely determined by $\Omega$ \cite{B2}. Its cohomology class $[\theta]\in H^1(M;\mathbb R)$ is called the \emph{Lee class}. If $\theta=0$, then $\Omega$ is symplectic. More generally, if $\theta$ is exact, say $\theta=df$, then
\[
d(e^{-f}\Omega)=0,
\]
and hence the LCS structure is globally conformally symplectic.

If $\Omega'=e^f\Omega$ for some smooth function $f$, then
\[
d\Omega'=(\theta+df)\wedge\Omega',
\]
so the corresponding Lee form is $\theta'=\theta+df$. In particular, the Lee class $[\theta]$ is invariant under conformal changes.

Associated with the Lee form is the twisted differential
\[
d_\theta=d-\theta\wedge.
\]
The LCS condition is equivalently written as $d_\theta\Omega=0$. The corresponding twisted, or Lichnerowicz, cohomology is an important tool in LCS geometry. In particular, Banyaga constructed examples of LCS forms which are not $d_\theta$-exact \cite{B1}. Related cohomological aspects of LCS structures have been studied further in \cite{AOT,LV}.

\subsection{LCS structures of the first and second kind}

Let $(M,\Omega,\theta)$ be an LCS manifold. We recall the distinction between LCS structures of the first and second kind; see \cite{B2,BM1}. Denote by
\[
\mathfrak X(M,\Omega)
=
\{X\in\mathfrak X(M)\mid \mathcal L_X\Omega=0\}
\]
the Lie algebra of infinitesimal automorphisms of the LCS structure. For every $X\in\mathfrak X(M,\Omega)$, the function $\theta(X)$ is constant on $M$. Hence the map
\[
\ell\colon\mathfrak X(M,\Omega)\longrightarrow\mathbb R,
\qquad
\ell(X)=\theta(X),
\]
is a Lie algebra homomorphism, called the \emph{Lee homomorphism} \cite{BM1}.

\begin{definition}\label{deffirstkind}
An LCS structure $(\Omega,\theta)$ is said to be \emph{of the first kind} if the Lee homomorphism $\ell$ is nonzero, equivalently, if it is surjective. It is said to be \emph{of the second kind} if $\ell$ is identically zero \cite{B2,BM1}.
\end{definition}

If an LCS structure is of the first kind, there exists $U\in\mathfrak X(M,\Omega)$ such that, after rescaling $U$, $\theta(U)=1$. Such a vector field is called an \emph{anti-Lee vector field} \cite{BM1}. In particular, the Lee form of an LCS structure of the first kind is nowhere vanishing. Consequently, if the Lee form has a zero, then the LCS structure is necessarily of the second kind.

\subsection{Ordinary LS category of LCS manifolds}

We record a simple observation showing that the ordinary LS category behaves quite differently for LCS manifolds than for symplectic manifolds.

\begin{proposition}\label{propcatlcs}
Let $n\geq2$. Then
\[
\{\cat(M)\mid M^{2n}\text{ is a closed LCS manifold}\}
=
\{3,\ldots,2n+1\}.
\]
Moreover, every value in this range is realized by an LCS manifold of the first kind.
\end{proposition}

\begin{proof}
Let $3\leq k\leq2n+1$ and set
\[
M_{n,k}=T^{k-2}\times S^{2n-k+2}
=S^1\times N_{n,k},
\qquad
N_{n,k}=T^{k-3}\times S^{2n-k+2}.
\]
We first show that the $(2n-1)$-dimensional manifold $N_{n,k}$ admits a contact form.

Suppose first that $k-3$ is even and $k<2n+1$. Then $N_{n,k}=T^{2a}\times S^{2b+1}$ for some $a\geq0$ and $b\geq1$. Starting from the standard contact structure on $S^{2b+1}$ and iterating Bourgeois' construction of contact structures on products with $T^2$, we obtain a contact structure on $N_{n,k}$ \cite{Bcontact}. If $k=2n+1$, then $N_{n,k}=T^{2n-1}$, which admits a contact structure by Bourgeois \cite{Bcontact}.

Suppose next that $k-3$ is odd. Then
\[
N_{n,k}=T^{2a}\times(S^1\times S^{2b})
\]
for some $a\geq0$ and $b\geq1$. The manifold $S^1\times S^{2b}$ admits the contact form
\[
\alpha_0
=
z\,ds+\sum_{j=1}^{b}(x_j\,dy_j-y_j\,dx_j),
\]
where $(x_1,y_1,\ldots,x_b,y_b,z)\in S^{2b}\subset\mathbb R^{2b+1}$. Iterating Bourgeois' construction again gives a contact structure on $N_{n,k}$. Thus in every case $N_{n,k}$ admits a contact form; let $\alpha$ be one.

If $t$ denotes the coordinate on the first $S^1$ factor of $M_{n,k}=S^1\times N_{n,k}$, set
\[
\theta=dt,
\qquad
\Omega=d\alpha-dt\wedge\alpha.
\]
Then
\[
d\Omega=dt\wedge d\alpha=\theta\wedge\Omega,
\qquad
\Omega^n=-n\,dt\wedge\alpha\wedge(d\alpha)^{n-1}\neq0.
\]
Hence $(\Omega,\theta)$ is an LCS structure on $M_{n,k}$. Moreover,
\[
\mathcal L_{\partial_t}\Omega=0,
\qquad
\theta(\partial_t)=1,
\]
so this LCS structure is of the first kind.

The standard cup-length lower bound and the product inequality for LS category give
\[
\cat(M_{n,k})
=
\cat\bigl(T^{k-2}\times S^{2n-k+2}\bigr)
=
k.
\]
Thus every integer $k$ with $3\leq k\leq2n+1$ is realized by a closed LCS manifold of the first kind.

It remains to exclude the values $1$ and $2$. A closed manifold of positive dimension is not contractible, and hence its LS category in our convention cannot be $1$. Moreover, in our unnormalized convention, a closed manifold of LS category $2$ is a homotopy sphere \cite{DKR}. An even-dimensional homotopy sphere of dimension at least four cannot admit an LCS structure. Indeed, since $H^1(M;\mathbb R)=0$, the Lee form of any LCS structure on $M$ would be exact, so the structure would be globally conformally symplectic. This would produce a symplectic form on $M$, which is impossible because $H^2(M;\mathbb R)=0$. Therefore $\cat(M)\geq3$ for every closed LCS manifold of dimension $2n\geq4$. Together with the dimensional upper bound $\cat(M)\leq2n+1$, this completes the proof.
\end{proof}

%% file: lscatoflee.tex
\section{Category and LCS structures of the first and second kind}

\subsection{LCS structures of the first kind}

As recalled in Section~2.3, the Lee form of an LCS structure of the first kind is nowhere vanishing. We therefore begin with the following general observation.

\begin{proposition}
Let $M$ be a closed manifold and let $\omega$ be a nowhere-vanishing closed one-form. Then
\[
\cat(M,[\omega])=0.
\]
\end{proposition}

\begin{proof}
Choose a Riemannian metric on $M$ and let $\omega^\sharp$ be the vector field dual to $\omega$. Since $\omega$ is nowhere vanishing, the vector field
\[
V=-\frac{\omega^\sharp}{|\omega^\sharp|^2}
\]
is well defined and satisfies $\omega(V)=-1$. Since $M$ is closed, the flow $\varphi_t$ of $V$ is defined for all $t\in\mathbb R$.

For each $N>0$, define
\[
h_t(x)=\varphi_{Nt}(x),
\qquad x\in M,\quad t\in[0,1].
\]
Then $h_0$ is the identity, and for the path $\gamma_x(t)=h_t(x)$ we have
\[
\int_{\gamma_x}\omega
=
\int_0^1 \omega(NV)\,dt
=
-N.
\]
Thus the whole manifold $M$ is $N$-movable with respect to $\omega$ for every $N>0$. Hence $\cat(M,[\omega])=0$.
\end{proof}

\begin{corollary}
Let $(M,\Omega,\theta)$ be a closed LCS manifold of the first kind. Then
\[
\cat(M,[\theta])=0.
\]
\end{corollary}

\begin{proof}
The Lee form of an LCS structure of the first kind is nowhere vanishing, so the preceding proposition applies.
\end{proof}

\subsection{LCS structures of the second kind}

The preceding vanishing result immediately yields an obstruction to the existence of an LCS structure of the first kind with a prescribed Lee class.

\begin{proposition}
Let $(M,\Omega,\theta)$ be a closed LCS manifold. If
\[
\cat(M,[\theta])>0,
\]
then $(\Omega,\theta)$ is of the second kind.
\end{proposition}

\begin{proof}
This follows immediately from the preceding corollary.
\end{proof}

Thus positivity of $\cat(M,[\theta])$ is an obstruction, depending only on the Lee class, to the existence of an LCS structure of the first kind with that Lee class. In the following sections, we obtain examples and cohomological lower bounds for this invariant, and later apply them to LCS blow-ups.

%% file: LCSmanifoldswithpositiveLeecategory.tex
\section{Examples with positive category}

\subsection{Products of surfaces}

We first consider products of closed orientable surfaces. Let
\[
M^{2k}=\Sigma_{g_1}\times\cdots\times\Sigma_{g_k},
\qquad k\geq2,
\]
where $g_i>1$ for every $i$. Let $\xi\in H^1(M;\mathbb R)$, and denote by $\xi_i$ the restriction of $\xi$ to the $i$-th factor. Set
\[
r(\xi)=\#\{i\in\{1,\ldots,k\}\mid \xi_i=0\}.
\]
Farber and Sch\"utz computed the category explicitly for such products \cite[Theorem~17]{FS1}.

\begin{theorem}[Farber--Sch\"utz]
Let $M^{2k}=\Sigma_{g_1}\times\cdots\times\Sigma_{g_k}$ with $g_i>1$. Then, for every $\xi\in H^1(M;\mathbb R)$,
\[
\cat(M,\xi)=1+2r(\xi).
\]
In particular, if $\xi_i\neq0$ for every $i$, then $\cat(M,\xi)=1$.
\end{theorem}

Since each $\Sigma_{g_i}$ is symplectic, the product $M$ admits a symplectic form. Let $0\neq\xi\in H^1(M;\mathbb R)$ and choose a closed one-form $\theta$ representing $\xi$. Since $\theta$ is not exact, the existence theorem of Bertelson and Meigniez \cite[Theorem~A]{BM2} gives an LCS form $\Omega$ satisfying
\[
d\Omega=\theta\wedge\Omega.
\]
Thus every nonzero cohomology class in $H^1(M;\mathbb R)$ can be realized as the Lee class of an LCS structure.

Combining this existence result with the computation of Farber and Sch\"utz gives the following.

\begin{corollary}
Let $M^{2k}=\Sigma_{g_1}\times\cdots\times\Sigma_{g_k}$, where $k\geq2$ and $g_i>1$, and let $0\neq\xi\in H^1(M;\mathbb R)$. Then $M$ admits an LCS structure with Lee class $\xi$, and
\[
\cat(M,\xi)=1+2r(\xi).
\]
Consequently, every LCS structure on $M$ whose Lee class is $\xi$ is of the second kind.
\end{corollary}

\begin{proof}
The existence of an LCS structure with Lee class $\xi$ follows from \cite[Theorem~A]{BM2}, and the formula for $\cat(M,\xi)$ follows from \cite[Theorem~17]{FS1}. Since $\cat(M,\xi)=1+2r(\xi)>0$, the result of the previous section implies that every LCS structure with Lee class $\xi$ is of the second kind.
\end{proof}

In dimension four, let $M=\Sigma_g\times\Sigma_h$, where $g,h>1$. For a nonzero class $\xi\in H^1(M;\mathbb R)$, there are two possibilities. If the restrictions of $\xi$ to both factors are nonzero, then $r(\xi)=0$ and $\cat(M,\xi)=1$. If $\xi$ vanishes on exactly one factor, then $r(\xi)=1$ and $\cat(M,\xi)=3$. Thus the category distinguishes different nonzero cohomology classes on the same underlying four-dimensional manifold.

\subsection{A four-dimensional example}

We next give a four-dimensional example which is not a product of surfaces. Let
\[
M_4=(S^1\times S^3)\#\mathbb{CP}^2\#\mathbb{CP}^2.
\]
We have
\[
H^1(M_4;\mathbb R)\cong\mathbb R,
\qquad
H^2(M_4;\mathbb Z)\cong\mathbb Ze_1\oplus\mathbb Ze_2,
\]
where the intersection form satisfies $e_1^2=e_2^2=1$ and $e_1e_2=0$. Moreover, $\chi(M_4)=2$ and $\sigma(M_4)=2$.

The second Stiefel--Whitney class satisfies
\[
w_2(M_4)\equiv e_1+e_2\pmod2.
\]
Consider
\[
c=e_1+3e_2\in H^2(M_4;\mathbb Z).
\]
Then $c\equiv w_2(M_4)\pmod2$ and
\[
c^2=1+9=10=2\chi(M_4)+3\sigma(M_4).
\]
By the standard criterion for almost complex structures on closed oriented four-manifolds \cite{GS}, $M_4$ admits an almost complex structure. In particular, it admits a nondegenerate two-form.

\begin{proposition}
For every nonzero class $\xi\in H^1(M_4;\mathbb R)$, there exists an LCS structure on $M_4$ with Lee class $\xi$. Moreover,
\[
\cat(M_4,\xi)>0,
\]
and consequently every LCS structure on $M_4$ whose Lee class is $\xi$ is of the second kind.
\end{proposition}

\begin{proof}
Choose a closed one-form $\theta$ representing $\xi$. Since $\xi\neq0$, the form $\theta$ is not exact. As $M_4$ admits a nondegenerate two-form, the theorem of Bertelson and Meigniez \cite[Theorem~A]{BM2} gives an LCS form $\Omega$ satisfying
\[
d\Omega=\theta\wedge\Omega.
\]
On the other hand, $\chi(M_4)=2\neq0$. Farber and Sch\"utz proved that
\[
\cat(X,\xi)=0\quad\Longrightarrow\quad\chi(X)=0
\]
\cite[Theorem~10]{FS1}. It follows that $\cat(M_4,\xi)>0$. The result of the previous section therefore implies that every LCS structure on $M_4$ with Lee class $\xi$ is of the second kind.
\end{proof}

Since $M_4$ is a closed four-manifold and $\xi\neq0$, the general upper bound of Farber and Sch\"utz \cite[Theorem~11 and Equation~(37)]{FS1} gives
\[
1\leq\cat(M_4,\xi)\leq3.
\]

%% file: lowerbound.tex
\section{Cohomological lower bounds}

In this section, we recall the cohomological lower bound for the category associated with a closed one-form due to Farber and Sch\"utz \cite{FS1}. This result will play a central role in the study of LCS blow-ups in the next section.

Let $X$ be a finite connected CW complex and let $\xi\in H^1(X;\mathbb R)$. Set
\[
H=H_1(X;\mathbb Z)/\ker\xi,
\]
where $\xi$ is regarded as a homomorphism $H_1(X;\mathbb Z)\to\mathbb R$. Let $\mathcal V_\xi$ denote the set of complex flat line bundles on $X$ whose monodromy is trivial on $\ker\xi$. For $L\in\mathcal V_\xi$, its monodromy induces a ring homomorphism
\[
\operatorname{Mon}_L\colon\mathbb Z[H]\longrightarrow\mathbb C.
\]

\begin{definition}[{\cite[Definition~6.1]{FS1}}]
A flat line bundle $L\in\mathcal V_\xi$ is called \emph{transcendental} if the homomorphism
\[
\operatorname{Mon}_L\colon\mathbb Z[H]\longrightarrow\mathbb C
\]
is injective.
\end{definition}

If $\xi\neq0$, transcendental flat line bundles in $\mathcal V_\xi$ exist; one may choose the monodromies of a basis of the free abelian group $H$ to be algebraically independent over $\mathbb Q$.

The following cohomological estimate is due to Farber and Sch\"utz.

\begin{theorem}[Farber--Sch\"utz {\cite[Theorem~7]{FS1}}]
Let $X$ be a finite CW complex and let $\xi\in H^1(X;\mathbb R)$. Suppose that $L\in\mathcal V_\xi$ is transcendental and that there exist classes
\[
v_0\in H^{d_0}(X;L),
\qquad
v_i\in H^{d_i}(X;\mathbb C),
\quad i=1,\ldots,k,
\]
with $d_i>0$ for $i=1,\ldots,k$, such that
\[
v_0\smile v_1\smile\cdots\smile v_k\neq0.
\]
Then
\[
\cat(X,\xi)>k.
\]
\end{theorem}

Since $\cat(X,\xi)$ is integer-valued, this is equivalent to $\cat(X,\xi)\geq k+1$. For $k=0$, we obtain the following useful criterion.

\begin{corollary}
Let $X$ be a finite CW complex and let $\xi\in H^1(X;\mathbb R)$. If there exists a transcendental flat line bundle $L\in\mathcal V_\xi$ such that
\[
H^*(X;L)\neq0,
\]
then
\[
\cat(X,\xi)>0.
\]
\end{corollary}

Applied to the Lee class of a closed LCS manifold, this gives a cohomological criterion for $\cat(M,[\theta])>0$, and hence, by Proposition~3.3, an obstruction to the existence of an LCS structure of the first kind with Lee class $[\theta]$.

Thus cohomology with flat local coefficients associated with the Lee class plays a role analogous to ordinary cohomology in the classical cup-length estimate. In the next section, we combine this lower bound with the topology of the exceptional divisor of an LCS blow-up to obtain quantitative estimates for the category of the pullback Lee class.

%% file: blowup.tex
\section{Blow-ups and LCS structures of the second kind}

In this section, we study LCS blow-ups from the viewpoint of the category associated with the Lee class. We first recall the LCS blow-up construction and an obstruction coming from the Euler characteristic. We then prove a quantitative lower bound for the category of the pullback Lee class using the exceptional divisor. Finally, we give a four-dimensional example in which the Euler characteristic of the blow-up vanishes, while the category still excludes LCS structures of the first kind with the prescribed Lee class.

\subsection{LCS blow-ups and the Euler characteristic}

Blow-up constructions also occur naturally in conformal symplectic geometry. Point blow-ups in the locally conformally K\"ahler setting were studied by Vuletescu \cite{V}. In the locally conformally symplectic setting, Chen and Yang proved that the blow-up of an LCS manifold at a point again admits an LCS structure \cite{CY}. This construction was subsequently extended by Yang, Yang and Zhao to blow-ups along suitable compact submanifolds \cite{YYZ}.

Let $(M^{2n},\Omega,\theta)$ be an LCS manifold, and let $i\colon Z\hookrightarrow M$ be a compact submanifold of real codimension $2r$, where $r\geq2$. Following Yang, Yang and Zhao \cite{YYZ}, the submanifold $Z$ is called an \emph{induced globally conformally symplectic submanifold}, or an \emph{IGCS submanifold}, if $i^*\Omega$ is nondegenerate and
\[
i^*[\theta]=0\in H^1(Z;\mathbb R).
\]

Yang, Yang and Zhao proved that the blow-up along such a submanifold again carries an LCS structure.

\begin{theorem}[Yang--Yang--Zhao {\cite{YYZ}}]
Let $(M,\Omega,\theta)$ be an LCS manifold, and let $i\colon Z\hookrightarrow M$ be a compact IGCS submanifold. If $\pi\colon\widetilde M\to M$ is the blow-up of $M$ along $Z$, then $\widetilde M$ admits an LCS structure whose Lee class is $\pi^*[\theta]$.
\end{theorem}

The exceptional divisor is $E=\mathbb P(N_{Z/M})$, which is a $\mathbb{CP}^{r-1}$-bundle over $Z$. Since $\chi(\mathbb{CP}^{r-1})=r$, multiplicativity of the Euler characteristic for fiber bundles gives $\chi(E)=r\chi(Z)$. Replacing a tubular neighborhood of $Z$ by its blow-up therefore gives
\[
\chi(\widetilde M)
=
\chi(M)-\chi(Z)+\chi(E)
=
\chi(M)+(r-1)\chi(Z).
\]

Since the Lee form of an LCS structure of the first kind is nowhere vanishing, and a closed manifold admitting a nowhere-vanishing one-form has vanishing Euler characteristic, we obtain the following.

\begin{proposition}
Let $\widetilde M$ be the blow-up of a closed LCS manifold $M$ along a compact IGCS submanifold $Z$ of real codimension $2r$, where $r\geq2$. If
\[
\chi(M)+(r-1)\chi(Z)\neq0,
\]
then $\widetilde M$ admits no LCS structure of the first kind.
\end{proposition}

\begin{proof}
The blow-up formula gives $\chi(\widetilde M)=\chi(M)+(r-1)\chi(Z)\neq0$. If $\widetilde M$ admitted an LCS structure of the first kind, its Lee form would be nowhere vanishing, and hence $\chi(\widetilde M)=0$, a contradiction.
\end{proof}

In particular, a point blow-up of a closed four-dimensional manifold satisfies $\chi(\widetilde M)=\chi(M)+1$. Thus, if $\chi(M)\neq-1$, the Euler characteristic alone excludes LCS structures of the first kind on the blow-up. If $\chi(M)=-1$, however, then $\chi(\widetilde M)=0$, and this obstruction disappears.

\subsection{A quantitative category obstruction}

We now obtain a lower bound for the category of the pullback Lee class. We first record a standard topological fact about blow-ups.

\begin{lemma}
Let $\pi\colon\widetilde M\to M$ be the blow-up of a manifold $M$ along a submanifold $Z$ of real codimension $2r\geq4$. Then
\[
\pi_*\colon\pi_1(\widetilde M)\longrightarrow\pi_1(M)
\]
is an isomorphism. Consequently,
\[
\pi_*\colon H_1(\widetilde M;\mathbb Z)
\longrightarrow
H_1(M;\mathbb Z)
\]
is an isomorphism.
\end{lemma}

\begin{proof}
Let $\nu Z$ be a tubular neighborhood of $Z$ and set $A=M\setminus\operatorname{Int}(\nu Z)$. The bundle $\partial(\nu Z)\to Z$ has fiber $S^{2r-1}$, which is simply connected since $r\geq2$, so $\pi_1(\partial(\nu Z))\cong\pi_1(Z)$. Moreover, $\nu Z$ deformation retracts onto $Z$.

Let $\widetilde{\nu Z}$ denote the blown-up tubular neighborhood. It deformation retracts onto the exceptional divisor $E=\mathbb P(N_{Z/M})$, and the projection $E\to Z$ has simply connected fiber $\mathbb{CP}^{r-1}$. Thus $\pi_1(\widetilde{\nu Z})\cong\pi_1(E)\cong\pi_1(Z)$.

The inclusions of the common boundary induce the corresponding identifications on fundamental groups. Applying the Seifert--van Kampen theorem to
\[
M=A\cup_{\partial(\nu Z)}\nu Z,
\qquad
\widetilde M=A\cup_{\partial(\nu Z)}\widetilde{\nu Z},
\]
shows that the blow-down map induces an isomorphism $\pi_1(\widetilde M)\cong\pi_1(M)$. The assertion for $H_1$ follows by abelianization.
\end{proof}

\begin{theorem}
Let $(M,\Omega,\theta)$ be a closed LCS manifold with $[\theta]\neq0$, and let $i\colon Z\hookrightarrow M$ be a nonempty compact IGCS submanifold of real codimension $2r$, where $r\geq2$. If $\pi\colon\widetilde M\to M$ is the blow-up along $Z$, then
\[
\cat(\widetilde M,\pi^*[\theta])\geq r-1.
\]
More generally, if there exist classes $u_1,\ldots,u_s\in H^{>0}(M;\mathbb C)$ such that
\[
i^*(u_1\smile\cdots\smile u_s)\neq0,
\]
then
\[
\cat(\widetilde M,\pi^*[\theta])\geq r-1+s.
\]
\end{theorem}

\begin{proof}
Set $\xi=[\theta]$ and choose a transcendental flat complex line bundle $L\in\mathcal V_\xi$. By the preceding lemma, the blow-down map induces an isomorphism
\[
\pi_*\colon H_1(\widetilde M;\mathbb Z)\longrightarrow H_1(M;\mathbb Z).
\]
Since $(\pi^*\xi)(a)=\xi(\pi_*a)$, the isomorphism $\pi_*$ maps $\ker(\pi^*\xi)$ onto $\ker\xi$ and therefore induces
\[
H_1(\widetilde M;\mathbb Z)/\ker(\pi^*\xi)
\cong
H_1(M;\mathbb Z)/\ker\xi.
\]
Under this identification, the monodromy homomorphism of $\pi^*L$ agrees with that of $L$. Hence $\pi^*L$ is transcendental with respect to $\pi^*\xi$.

Let $E=\mathbb P(N_{Z/M})$ be the exceptional divisor, let $j\colon E\hookrightarrow\widetilde M$ be the inclusion, and let $p\colon E\to Z$ be the natural projection. Since $Z$ is IGCS, $i^*\xi=0$. Hence the monodromy of $L$ is trivial on the image of $H_1(Z;\mathbb Z)$ in $H_1(M;\mathbb Z)$, and therefore $i^*L$ is trivial. Consequently, $j^*\pi^*L=p^*i^*L$ is also trivial.

Choose a nowhere-vanishing flat section $1_L\in H^0(E;j^*\pi^*L)$. Using the trivialization determined by $1_L$, we identify
\[
H^*(E;j^*\pi^*L)\cong H^*(E;\mathbb C).
\]

The normal bundle of the exceptional divisor is the tautological complex line bundle
\[
N_{E/\widetilde M}\cong\mathcal O_E(-1).
\]
In particular, it is canonically oriented, and hence there is a Gysin homomorphism with local coefficients
\[
j_!\colon H^*(E;j^*\pi^*L)
\longrightarrow
H^{*+2}(\widetilde M;\pi^*L).
\]
Put $h=c_1(\mathcal O_E(-1))\in H^2(E;\mathbb C)$. We use the projection and self-intersection formulas for the Gysin homomorphism; see \cite[Section~5.1, Proposition~5.4]{M}. Define
\[
v_0=j_!(1_L)\in H^2(\widetilde M;\pi^*L),
\qquad
e=j_!(1)\in H^2(\widetilde M;\mathbb C).
\]

By the projection formula,
\[
v_0\smile e^q
=
j_!\bigl(1_L\smile j^*(e^q)\bigr).
\]
Since the self-intersection formula gives $j^*e=j^*j_!(1)=h$, we obtain
\[
v_0\smile e^q=j_!(h^q),
\qquad
j^*(v_0\smile e^q)=h^{q+1}.
\]
For $0\leq q\leq r-2$, the restriction of $h^{q+1}$ to each fiber $\mathbb{CP}^{r-1}$ of $p\colon E\to Z$ is nonzero. Therefore $v_0\smile e^q\neq0$, and in particular
\[
v_0\smile e^{r-2}\neq0.
\]

Applying the Farber--Sch\"utz estimate \cite[Theorem~7]{FS1} to $v_0$ together with $r-2$ copies of the positive-degree class $e$ gives
\[
\cat(\widetilde M,\pi^*\xi)\geq r-1.
\]

Now suppose that $u_1,\ldots,u_s\in H^{>0}(M;\mathbb C)$ satisfy
\[
a=i^*(u_1\smile\cdots\smile u_s)\neq0.
\]
The projection formula gives
\[
v_0\smile e^{r-2}\smile\pi^*u_1\smile\cdots\smile\pi^*u_s
=
j_!\left(h^{r-2}\smile p^*a\right).
\]
Restricting this class to $E$ and applying the self-intersection formula gives $h^{r-1}\smile p^*a$. By the projective bundle theorem \cite{BT},
\[
H^*(E;\mathbb C)
=
H^*(Z;\mathbb C)\{1,h,\ldots,h^{r-1}\}
\]
as an $H^*(Z;\mathbb C)$-module. Hence $a\neq0$ implies $h^{r-1}\smile p^*a\neq0$, and therefore
\[
v_0\smile e^{r-2}\smile\pi^*u_1\smile\cdots\smile\pi^*u_s
\neq0.
\]
Applying the Farber--Sch\"utz estimate to the $r-2+s$ positive-degree classes gives
\[
\cat(\widetilde M,\pi^*\xi)\geq r-1+s.
\]
\end{proof}

The basic estimate already gives the following class-wide obstruction.

\begin{corollary}
Under the assumptions of the preceding theorem, the class $\pi^*[\theta]$ has no nowhere-vanishing closed one-form representative. Consequently, every LCS structure on $\widetilde M$ whose Lee class is $\pi^*[\theta]$ is of the second kind.
\end{corollary}

\begin{proof}
Since $r\geq2$, we have $\cat(\widetilde M,\pi^*[\theta])\geq r-1>0$. If $\alpha$ were a nowhere-vanishing closed one-form representing $\pi^*[\theta]$, Proposition~3.1 would give $\cat(\widetilde M,[\alpha])=0$, a contradiction. The second assertion follows because the Lee form of an LCS structure of the first kind is nowhere vanishing.
\end{proof}

For point blow-ups, the estimate takes a particularly simple form.

\begin{corollary}
Let $(M^{2n},\Omega,\theta)$ be a closed LCS manifold with $[\theta]\neq0$, and let $\pi\colon\widetilde M\to M$ be the blow-up at a point. Then
\[
\cat(\widetilde M,\pi^*[\theta])\geq n-1.
\]
Consequently, every LCS structure on $\widetilde M$ whose Lee class is $\pi^*[\theta]$ is of the second kind.
\end{corollary}

\begin{proof}
A point is an IGCS submanifold of real codimension $2n$, so the preceding theorem applies with $r=n$.
\end{proof}

\begin{remark}
For the particular LCS structure arising from the blow-up construction, one may choose a representative of the Lee class which vanishes in a neighborhood of the center \cite{YYZ}. Its pullback therefore vanishes near the exceptional divisor, so this particular structure is directly seen to be of the second kind. The preceding corollary is stronger: it depends only on the cohomology class $\pi^*[\theta]$ and applies to every LCS structure with that Lee class.
\end{remark}

\subsection{An example with vanishing Euler characteristic}

We conclude with an example in which the Euler characteristic of the blow-up vanishes but the category obstruction remains nontrivial. Let
\[
M=(S^1\times S^3)\#(S^1\times S^3)\#\mathbb{CP}^2.
\]
A direct computation gives $\chi(M)=-1$, $\sigma(M)=1$, $H^1(M;\mathbb R)\cong\mathbb R^2$, and
\[
H^2(M;\mathbb Z)\cong\mathbb Ze,
\qquad
e^2=1.
\]
The class $e$ is characteristic, so $e\equiv w_2(M)\pmod2$, and
\[
e^2=1=2\chi(M)+3\sigma(M).
\]
By the criterion for almost complex structures on closed oriented four-manifolds \cite{GS}, $M$ admits an almost complex structure and hence a nondegenerate two-form.

Let $0\neq\xi\in H^1(M;\mathbb R)$. By the existence theorem of Bertelson and Meigniez \cite[Theorem~A]{BM2}, $M$ admits an LCS structure whose Lee class is $\xi$. Let $\pi\colon\widetilde M\to M$ be the blow-up at a point. Since $M$ is four-dimensional,
\[
\chi(\widetilde M)=\chi(M)+1=0.
\]
Thus the Euler characteristic gives no obstruction to the existence of an LCS structure of the first kind on $\widetilde M$. On the other hand, the point blow-up estimate gives
\[
\cat(\widetilde M,\pi^*\xi)\geq1.
\]
Consequently, every LCS structure on $\widetilde M$ whose Lee class is $\pi^*\xi$ is of the second kind.

Thus the category obstruction remains nontrivial even when the Euler characteristic obstruction vanishes. In particular, $\pi^*\xi$ cannot be represented by a nowhere-vanishing closed one-form.